\documentclass[a4paper]{article}
\usepackage{fullpage}
\usepackage{authblk}

\usepackage[T1]{fontenc}
\usepackage[utf8]{inputenc}
\usepackage[english]{babel}

\usepackage{amssymb}
\usepackage{amsthm}
\usepackage{amsmath}
\usepackage{amsfonts}
\usepackage{cite}

\usepackage{color}
\usepackage{siunitx}

\newtheorem{remark}{Remark}
\newtheorem{theorem}{Theorem}
\newtheorem{definition}{Definition}
\newtheorem{prop}{Proposition}

\def\eps{\varepsilon}

\def\St{\te{\Sigma}}
\def\vvt{\ve{\varphi}}
\def\intd{\int_\Omega}

\def\intt{\int_0^T}
\def\I{\te{I}}
\def\T{\te{T}}
\def\D{\te{D}}
\def\W{\te{W}}
\def\div{\mathrm{div\,}}
\def\tr{\mathrm{tr\,}}
\newcommand\pd[2]{\frac{\partial #1}{\partial #2}}

\newcommand{\ve}[1]{\boldsymbol{#1}}
\newcommand{\te}[1]{\boldsymbol{#1}}

\newcommand{\vv}{\te{v}}
\def\gradv{\nabla \vv}
\def\ddt{\frac{\mathrm{d}}{\mathrm{dt}}}
\def\pdt{\frac{\partial}{\partial t}}

\def\dx{\,\mathrm{d}x}
\def\dt{\,\mathrm{d}t}
\def\grad{\nabla}
\def\S{\te{S}}
\def\q{\ve{q}}
\def\rho{\varrho}
\def\12{\frac{1}{2}}
\newcommand\be[2]{\begin{equation}\label{#1}#2 \end{equation}}

\title{Stress-diffusive regularizations\\ of non-dissipative rate-type materials }

\author[1,2]{Jan Burczak}
\author[3]{Josef M\'alek}
\author[4,5]{Piotr Minakowski}

\affil[1]{Institute of Mathematics, Polish Academy of Sciences, \'Sniadeckich 8, 00-656 Warsaw, Poland
}
\affil[2]{OxPDE, Mathematical Institute, University of Oxford, Andrew Wiles Building, Radcliffe Observatory Quarter, Woodstock Road, Oxford OX2 6GG, United Kingdom
}
\affil[3]{Charles University, Faculty of Mathematics and Physics, Mathematical Institute, Sokolovsk\'{a} 83, 186~75 Prague 8, Czech Republic}
\affil[4]{Heidelberg University, Interdisciplinary Center for Scientific Computing, Im Neuenheimer Feld 205, 69120 Heidelberg, Germany}
\affil[5]{University of Warsaw, Institute of Applied Mathematics and Mechanics,  Banacha 2, 02-097 Warsaw, Poland}

\begin{document}
\maketitle

\centerline{\emph{To Tom\'{a}\v{s} Roub\'\i{}\v{c}ek on the occasion of his 60th birthday.}}
\bigskip

\begin{abstract}

We consider non-dissipative (elastic) rate-type material models that are derived within the Gibbs-potential-based thermodynamic framework. 
Since the absence of any dissipative mechanism in the model prevents us from establishing even a local-in-time existence result in 
two spatial dimensions for a spatially periodic problem, we propose two regularisations. For such regularized problems we obtain well-posedness 
of the planar, spatially periodic problem. In contrast with existing results, we prove ours for a regularizing term present solely in the 
evolution equation for the stress.
\end{abstract}
	
\section{Introduction}\label{sec1}

Elastic materials are bodies that are not capable of producing entropy or, in a purely mechanical context, of dissipating energy. 
Due to this characterization they are called non-dissipative materials.

Starting from this thermodynamic point of view and from the assumption that the mechanism in
which a material stores energy
is encoded into the constitutive equation for the Gibbs potential, whereby the Gibbs potential is a function of the Cauchy stress, 
Rajagopal and Srinivasa have in a series of papers (see in particular \cite{Rajagopal2009, Rajagopal2011}) extended the framework 
of elasticity to rate-type materials; see Rajagopal \cite{RajA, RajB, RajC} for further details including the references and comments to earlier achievements, in particular to the concept of hypoelasticity introduced in
Truesdell \cite{Truesdell1955}. 
Besides providing a new class of non-dissipative bodies, the advantage of this approach lies in the fact that it only uses quantities 
defined in the current configuration. Consequently it does not require introducing notions of a reference state or any type of 
strain. Hence a fully Eulerian theory of elasticity is applicable, for example, to the processes concerning biological matter where, 
due to the fact that cells are born and die, it is reasonable to consider only quantities at a current time and at a given position.

This fully Eulerian Gibbs-potential-based thermodynamic approach has been further extended to describe the response of visco-elastic materials, 
see Rajagopal and Srinivasa \cite{Rajagopal2011}, or to model severe plastic deformations of a crystalline solid treated as a material flow through an adjustable crystal lattice space, see Kratochv\'\i{}l et al. \cite{Minak2016}.

Our original intention has been to develop a mathematical theory for initial and boundary-value problems involving such a class of elastic 
(non-dissipative) models. To be more specific, restricting ourselves to materials where the density is uniform and considering only 
isothermal processes, we wish to analyse, in a $d$-dimensional domain $\Omega$, the following set of partial differential equations
(PDEs):
\begin{subequations}\label{syso}
\begin{align}
\div\vv &= 0, \label{sysoa}\\
\pd{\vv}{t} + (\vv\cdot\grad)\vv   + \nabla p &= \div\S, \label{sysob}\\
\pd{\S}{t} + (\vv \cdot \nabla)\S + \S\W -\W\!\S  &= \D,\label{sysod}
\end{align}
\end{subequations}
where $\vv = (v_1,\dots,v_d)$ stands for the velocity, $p$ for the spherical part of the Cauchy stress (the pressure), $\S = (S_{ij})_{i,j=1}^d$ 
for the deviatoric part of the Cauchy stress\footnote{Divided by the constant density.} that is supposed to be symmetric ($\S=\S^T$), 
$\D$ and $\W$ stand for the symmetric and antisymmetric parts of the velocity gradient, i.e. by definition $\D = \12(\gradv+\gradv^T)$ 
and $\W = \12(\gradv-\gradv^T)$, respectively. The symbol $(\vv\cdot\grad)$ signifies the operator $\sum_{k=1}^{d} v_k \frac{\partial}{\partial x_k}$.

In order to understand the basic mathematical features of \eqref{syso}, we eliminate the influence of the boundary by assuming that $\Omega$ is a 
periodic cell and by considering $\vv$, $p$, $\S$ that are $\Omega-$periodic. Since\footnote{The energy identity \eqref{ener} is obtained by adding 
the result of scalar multiplication of \eqref{sysob} by $\vv$ to the result of scalar multiplication of \eqref{sysod} by $\S$, see also 
Sections \ref{1stae} and \ref{app:md} for details.} 
\be{ener}{\|\vv(t)\|_{L^2}^2 + \|\S(t)\|_{L^2}^2 = \|\vv(0)\|_{L^2}^2 + \|\S(0)\|_{L^2}^2 \quad \textrm{ for all } t\in[0,T]} 
and in addition\footnote{This feature has been successfully exploited in \cite{Lions2000} for establishing the first global-in-time existence 
result for large data with respect to rate-type visco-elastic fluid models in three-dimensions.} 
\be{M100}{(\S\W-\W\!\S):\S = 0, } one may, at the first glance, pose a conjuncture that the existence theory for the Euler equation (obtained 
formally by setting $\S=\te{O}$ in \eqref{sysob}), as developed for example in Kato \cite{Kato1967}, can be successfully extended to \eqref{syso}. 

As indicated in the Appendix, this approach to developing local-in-time existence theory seems to be inapplicable to \eqref{syso} even in two spatial 
dimensions. Consequently we leave this conjecture as an interesting open problem and propose to study two different regularisations obtained by 
adding the terms $-\eps\Delta \pdt \S$ or $- \eps\Delta \S$ to \eqref{sysod}: 
$$
\pd{\S}{t} + (\vv \cdot \nabla)\S + \S\W -\W\!\S -\eps \Delta \pdt \S = \D  \quad \text{ or } \quad \pd{\S}{t} + (\vv \cdot \nabla)\S + \S\W -\W\!\S - \eps \Delta  \S = \D.
$$

For the first regularization, we observe that, instead of \eqref{ener}, one has
\be{enerI}{\|\vv(t)\|_{L^2}^2 + \|\S(t)\|_{L^2}^2 + \eps \|\nabla \S (t) \|_2^2 = \|\vv(0)\|_{L^2}^2 + \|\S(0)\|_{L^2}^2 + \eps \|\nabla \S (0) \|_2^2 
\quad \textrm{ for all } t\in[0,T].}
This information turns out to be strong enough for establishing the long-time existence and uniqueness of a weak solution possessing certain higher 
regularity. We are not aware of any physical meaning of this type of regularization. 

The second, 
weaker regularization leads to 
\be{enerII}{\|\vv(t)\|_{L^2}^2 + \|\S(t)\|_{L^2}^2 + 2 \eps \int_0^t \|\nabla \S (t) \|_2^2 \dt = \|\vv(0)\|_{L^2}^2 + \|\S(0)\|_{L^2}^2 \quad 
\textrm{ for all } t\in[0,T],} 
and it suffices for a short-time existence result or a global existence result for small initial data. A physical justification for the diffusive 
regularizing term can be found in the literature. 
For example, as was pointed out in \cite{Barrett2011, Suli2011, Lu2016}, an Oldroyd–B type model with stress diffusion can be derived 
from a Navier–Stokes–Fokker–Planck system arising in the kinetic theory of dilute polymeric fluids, where polymer chains immersed 
in a barotropic, incompressible, isothermal, viscous Newtonian solvent, are idealized as pairs of massless beads connected with 
Hookean springs. Moreover, non-dimensionalization leads to the conclusion that the dissipation parameter $\eps$ takes the values 
in the interval $(10^{-9}, 10^{-7})$ and is thus almost negligible.

Besides the goal to identify \eqref{syso} as an interesting model of elasticity worthy of further mathematical investigation, 
the aim of this paper is to show well-posedness for these two regularized problems. 

The paper is organised as follows. We first recall, still in Section \ref{sec1}, the derivation of \eqref{syso} based on a Gibbs-potential-based 
thermodynamical framework. We also provide a brief overview regarding the PDE analysis of rate-type visco-elastic models. 
Then, in Section \ref{sec2} we formulate Theorem 1 and Theorem 2 concerning well-posedness of the regularized problems considered. 
We prove these results in subsequent sections. 

\subsection{Gibbs-potential-based thermodynamic derivation of \eqref{syso}}

Let a body, considered at the current instant $t$, be identified with a bounded open set $\Omega \subset \Re^d$. The position of any particle at the current 
instant is denoted by $\ve{x}$ and its velocity by $\vv$. The mass density of the material is denoted by $\varrho$ and the Cauchy 
stress by $\T$. The governing balance equations for mass, linear and angular momenta (in the absence of body forces) and energy 
(in the absence of heat sources) as well as the formulation of the second law of thermodynamics take the following form:
\begin{subequations}\label{sysn}
\begin{eqnarray}
\dot \varrho +\varrho\, \text{div\,}\vv &=& 0,\label{mass} \\
\varrho\, \dot{\vv}  &=&  \text{div\,}\,\T, \qquad \T=\T^T\,, \label{balance} \\
\varrho\, \dot{\epsilon} &=& \T:\D - \text{div}\,\ve{\tilde{q}}, \label{energy_balance} \\
\varrho\zeta \!\!&:=& \varrho \dot{\eta} + \text{div} \left(\frac{\ve{\tilde{q}}}{\theta}\right) \quad \textrm{ and } \quad \zeta\ge 0, \label{entropy_balance}
\end{eqnarray}
\end{subequations}
where the material time derivative of a scalar function $z$ is given by $\dot{z}= \pdt{z} + (\vv \cdot\nabla) z $ (for a vector and tensor-valued 
function, the same relation is applied to each component). In the above equations, $\epsilon$ stands for the specific internal energy, 
$\ve{\tilde{q}}$ for the heat flux vector, $\eta$ for the specific entropy, $\theta$ for the temperature and $\zeta$ for the specific 
rate of entropy production; here we tacitly assume that the entropy flux is of the form $\ve{\tilde{q}}/\theta$. 

We shall consider incompressible materials with uniform density, i.e.,
$$\div\vv=\tr\D =0 \text{ and } \varrho \text{ is constant.}$$
Decomposing the Cauchy stress as 
$$\T = -\tilde{p}\I + \T_\delta, \text{ where } \tilde{p}:= -\frac{1}{d}\tr \T,$$
and setting $\q:=\tilde{\q}/\varrho$, $p:=\tilde{p}/\varrho$ and $\S := \T_\delta/\varrho$,
the governing system of equations \eqref{mass}-\eqref{energy_balance} reduces to 
\begin{equation}\label{sysn2}
\text{div\,}\vv = 0,\quad \dot{\vv}  = -\nabla {p} + \text{div}\,\S,\quad \dot\epsilon = \S:\D -\div\q.
\end{equation}

Next, let us introduce the specific Helmholtz free energy $\psi$ and the specific rate of dissipation  $\xi$ through 
$$\psi := \epsilon - \theta \eta, \quad\textrm{ and } \quad \xi := \theta\zeta.$$
With this notation, the equations \eqref{entropy_balance} and \eqref{sysn2}$_3$ lead to the following equation for the rate of dissipation:
\begin{equation}
\xi = \S:\D - \dot{\psi} - \eta \dot{\theta} - \ve{q}\cdot \frac{\nabla\theta}{\theta} \quad \textrm{ and } \quad \xi\ge 0. \label{entropy_balance_2}
\end{equation}
 
Following the Gibbs-potential-based thermodynamic framework as developed by Rajagopal and Srinivasa in \cite{Rajagopal2011}, we assume 
that the specific Gibbs potential, denoted by $G$, is a function of the temperature $\theta$ and $\S$, i.e.,
\begin{equation}
 G(t,x) = {G}(\theta(t,x), \S(t,x)) \qquad  \textrm{ or briefly } \qquad G = {G}(\theta, \S) \label{Gib_1}
\end{equation}
We also require that the Helmholtz free energy, the internal energy and the entropy, considered as functions of $\theta$ and $\S$, satisfy
\begin{equation}
\begin{split}
\psi(\theta, \S) &= G(\theta, \S) - \frac{\partial G (\theta,\S)}{\partial \S} : \S, \\
\epsilon(\theta, \S) &= G(\theta, \S) - \frac{\partial G(\theta, \S)}{\partial\S}:\S - \frac{\partial G(\theta, \S)}{\partial\theta}\theta, \\
\eta (\theta, \S) &= - \frac{\partial G(\theta, \S)}{\partial\theta}.
\end{split}\label{int_energy}
\end{equation}
Inserting the first and third of these relations into \eqref{entropy_balance_2}, we obtain 
\begin{equation}
\xi = \S:\left\{\D + \frac{\partial^2 G}{\partial \S^2}\dot{\S} + \frac{\partial^2 G}{\partial \S \partial \theta}\dot{\theta} \right\} - \ve{q}\cdot \frac{\nabla\theta}{\theta}\quad \textrm{ and } \quad \xi\ge 0. \label{entropy_balance_3}
\end{equation}

In what follows, we restrict ourselves to isothermal processes. Then the equation \eqref{entropy_balance_3} reduces to  
\begin{equation}
\xi = \S:\D + \S:\frac{\partial^{2} G}{\partial\S^{2}}\dot{\S} \quad \textrm{ and } \quad \xi\ge 0. \label{dissipation}
\end{equation}
We thus arrive at a representation of thermodynamics associated with the specification of the Gibbs potential (as given in \eqref{Gib_1}). 
The achieved form of \eqref{dissipation} has, however, the following defficiency: while $\D$ and $\S$ are both objective tensors, 
$\dot{\S}$ and consequently $({\partial^{2} G})/{(\partial\S^{2})}\, \dot{\S}$ are not objective tensors.

In \cite{Rajagopal2011}, Rajagopal and Srinivasa propose two approaches  to overcome this difficulty. While the second one is more general 
and provides a possibility to include anisotropic responses, we shall recall the first approach here, as it is the simplest way for completing 
the derivation of the system \eqref{syso} considered.

Let us first consider a particular form of the Gibbs potential, namely $G(\S) = -\12 |\S|^2$. Then \eqref{dissipation} simplifies to
\begin{equation}
\xi = \S:(\D - \dot{\S}) \quad \textrm{ and } \quad \xi\ge 0. \label{dissipation_new}
\end{equation}
Using the orthogonality condition \eqref{M100}, we easily observe that \eqref{dissipation_new} can be rewritten as 
\begin{equation}
\xi = \S: (\D - \dot{\S} - \S\W + \W\!\S) \quad \textrm{ and } \quad \xi\ge 0. \label{dissipation_new2}
\end{equation}
A remarkable difference between \eqref{dissipation_new} and \eqref{dissipation_new2} is that the
term $\dot{\S} + \S\W - \W\!\S$ in \eqref{dissipation_new2} is objective while $\dot{\S}$ in \eqref{dissipation_new} is not. 

Requiring further that the dissipation rate $\xi$ in \eqref{dissipation_new2} vanishes for arbitrary $\S$, we obtain 
\begin{equation*}
\dot{\S} + \S\W - \W\!\S = \D, \end{equation*}
which is \eqref{sysod}. The other governing equations, namely \eqref{sysoa} and \eqref{sysob}, are stated in \eqref{sysn2}.

\subsection{An overview of known results}\label{sec:results}

As follows from the above derivation, there are no dissipative terms present in \eqref{syso}. Consequently, the structure of the equation 
\eqref{sysob} seems identical to the Euler equations with the external force $\div\S$. Therefore, the results regarding the solvability 
of the Euler equation might be important in the context of the analysis of our problem. Unfortunately, the available local-in-time 
existence and uniqueness results for the Euler equation (see in particular \cite{Lichtenstein1925, Gunther1927, Wolibner1933, Kato1967, Pulvirenti1994}) 
do not seem to be applicable to \eqref{syso} due to the fact that the right-hand side of \eqref{sysob} is not regular enough. The difficulties connected with this approach are presented in Appendix.

Alternatively, one could follow a recent approach developed by DeLellis and Sz\'{e}kelyhidi (see \cite{LellisSzekehylidi2011} and \cite{LellisSzekehylidi2013}), 
based on the convex integration and Baire's  category principle, that provides the global-in-time existence of (infinitely many) weak solutions to the Euler 
system for a subset of initial data that is dense in $L^2(\Omega)_{div}$, see in particular Wiedemann \cite{Wiedemann2011}. This result has been strengthened 
by Chiodaroli, Feireisl and Kreml in \cite{ChioEFKreml2015} who considered the compressible Euler-Fourier system and proved that for arbitrary smooth positive 
initial density and temperature there is a bounded initial velocity so that the considered initial spatially-periodic problem admits infinitely many weak 
solutions that emanate from this fix set of initial data and satisfy the first law of thermodynamics (conservation of energy). 
Such results are thus closely related not only to the original system \eqref{syso} but also to its regularization by $-\eps \Delta\S$ 
studied in this paper. 

\smallskip

Regarding available analytical studies concerning weak solutions to stress-diffusive models, it is worth noting that all of them concern
systems where the balance of linear momentum contains additional diffusion of the type $-\Delta \vv$. More specifically, the existence of a global weak solution to the Oldroyd-B model with stress diffusion was proved in two space dimensions 
by Barrett and Boyaval \cite{Barrett2011} (see also Barrett and S\"uli \cite{Suli2011} or Luk\'a\v{c}ov\'a-Medvi\v{d}ov\'a et al. \cite{Mizerova2015}).  
Regularity of solutions of the Oldroyd-B equations in two spatial dimensions with spatial diffusion of the polymeric stress tensor
have been proved in Constantin and Kliegl \cite{Constantin2012}, where the authors take advantage of the nonnegativity of the polymeric stress matrix, 
which is preserved under diffusive evolution. Recently Chupin and Martin \cite{Chupin2015}, addressed the stationary Oldroyd-B model 
with a diffusive stress, from both an analytical and a numerical perspective. The authors investigated, by means of numerical simulations, 
the behaviour of the model with respect to vanishing diffusion, and concluded that solutions of the diffusive model converge to solutions 
of the non-diffusive model at order $1$ in the $W^{1,2}$ norm. Moreover, numerical stability of the effect of including the stress-diffusive 
term into the classical Oldroyd-B constitutive equation has been studied in \cite{Beris1995}. 

Let us re-emphasize that in the references mentioned above the authors take advantage of the presence of regularizing terms  both in the 
momentum equation and in the evolution equation for $\S$. In contrast, our results  require the regularization only in the equation 
for the stress.

For the sake of completeness, let us provide an overview of results concerning existence of solutions to visco-elastic fluids models, in 
particular to the Oldroyd-B model. There are several classes of visco-elastic fluid models that differ from our model by the presence of 
the dissipative term in the balance of linear momentum (typically in the form $-\nu \Delta \vv$), by the different form of the objective 
derivative and by the presence of other terms.

To the authors' knowledge, the first result on incompressible Oldroyd-B fluids was obtained by Guillopé and Saut \cite{Guillope1990}. The 
result concerns local-in-time existence of regular solutions as well as existence of global-in-time solutions for small initial data in a 
Hilbert framework. The main obstacle to obtaining existence results \emph{in the large} was the fact that, in general, there is no appropriate 
energy estimate for such a non-Newtonian fluid. (As a review paper in this direction, we refer to Fernández-Cara, Guillén and Ortega \cite{Cara2002}.) 
Despite this difficulty, Lions and Masmoudi established in \cite{Lions2000} existence of global weak solutions for a model with the 
Zaremba-Jaumann derivative. This seems to be one of the most significant results in this area. The authors use essentially that 
additional energy estimates are available for the Zaremba-Jaumann objective time derivative. The result by Lions and Masmoudi 
was generalised by Bejaoui and Majdoub in \cite{Majdoub2013}, where the authors replaced the Laplacian term by $\div(f(\D))$ 
with a tensorial function $f$, which is $C^1$, monotone, coercive and enjoys a $p$-growth with $p\ge 2$ in two dimensions ($p\ge 5/2$ for $d=3$).    

For well-posedness results in scaling-invariant Besov spaces, we refer to the work of Chemin and Masmoudi \cite{Chemin2001}, 
where they also provide certain blow-up criteria, both for two and three dimensions. Further interesting results concerning the
local well-posedness of the initial-boundary-value problem for Oldroyd-type fluids have been obtained in several other studies, 
see Liu et al. \cite{Liu2005} or Liu et al. \cite{Liu2008}.

Results for the compressible Oldroyd-B model are much scarcer. Lei \cite{Lei2006} proved the local and global existence of classical solutions 
for a compressible Oldroyd-B system in a torus with small initial data. He also studied the incompressible limit problem and showed that 
compressible flows with well-prepared initial data converge to incompressible ones when the Mach number converges to zero. Strong solutions of three-dimensional flows of compressible Oldroyd-B fluids were studied in  Fung~and~Zi~\cite{Fang2013}. Recently, Barrett et al. \cite{Lu2016} established long-time and large-data existence of weak solutions to compressible Oldroyd-B fluids with stress diffusion.

All of the results mentioned above take advantage of the presence of the Newtonian stress tensor in the balance of linear momentum and, 
consequently, of the boundedness of the velocity gradient in a Lebesgue space (typically $L^2$). 
Such a piece of information however does not follow from the first a-priori estimates for the systems considered here.

\section{Main result}\label{sec2}
In what follows we set
\begin{equation*}
T\in (0,\infty) \, \textrm{ and } \, \Omega=(0, L_1) \times (0, L_2) \qquad (L_i\in (0, \infty),\, i=1,2).
\end{equation*}

We will assume that all functions  considered are spatially $\Omega$-periodic and that their mean values over $\Omega$ vanish. 
For spatially $\Omega$-periodic functions, we employ standard notation for the function spaces considered, see for example \cite{Temam2001, MaRa2005} 
for appropriate definitions of $\Omega$-periodic function spaces.

\subsection[
Regularization by $-\eps \Delta \pdt \S$. Global-in-time existence of a weak solution]{ 
Regularization by $-\eps \Delta \pdt \S$: global-in-time existence}

Let us first consider the system of partial differential equations model \eqref{syso} regularized by adding the term $-\eps \Delta \pdt \S$ 
to the left-hand side of \eqref{syso}$_3$. Since we are unable to pass to the limit with $\eps \to 0$, for the sake of brevity we 
set $\eps = 1$ in what follows.

The formulation of the problem is thus the following: for given $\Omega$-periodic initial data $\vv_0, \S_0$, 
find $\vv(x,t):\Omega \times (0,T) \mapsto \Re^2$, \mbox{$p(x,t):\Omega \times (0,T) \mapsto \Re$}, and 
$\S(x,t):\Omega \times (0,T) \mapsto \Re^{2\times 2}_{\text{sym}}$ that are $\Omega$-periodic and satisfy
\begin{subequations}\label{sysr1}
\begin{align}
\div\vv &= 0, \label{sysa}\\
\pd{\vv}{t} + (\vv\cdot\nabla)\vv   + \nabla p &= \div\S, \label{sysb}\\
\pd{\S}{t} + (\vv \cdot \nabla)\S + \S\W -\W\!\S - \Delta \pdt \S &= \D\label{sysc},\\
\vv(x,0) &= \vv_0,\\
\S(x,0) &= \S_0.
\end{align}
\end{subequations}

Let us now specify our notion of a weak solution of \eqref{sysr1}.  By $\langle \cdot,\cdot \rangle$ we will denote the duality 
pairing between ${(W^{1,2}_{div})^\ast}$ and $W^{1,2}_{div}$ for $\vv$ or between $(W^{1,2})^\ast$, $W^{1,2}$ for $\S$.

\begin{definition}\label{def:cr}
We say that a couple $(\vv,\S)$ satisfying 
$$\vv \in L^2 (0, T ; W^{1,2}_{div})\cap \mathcal{C}([0,T];L^2_{div}), \quad \S \in  L^2(0,T;W^{1,2})\cap \mathcal{C}([0,T];L^2),$$
$$ \pd{\vv}{t} \in L^2(0,T;(W_{\div}^{1,2})^\star), \quad \ \pd{\S}{t} \in L^2(0,T;L^{2})$$
is a weak solution to the problem \eqref{sysr1} if for all $\Omega$-periodic $ \vvt \in C^\infty(\Omega\times (-\infty, T])^2$ such 
that  $\div \vvt = 0$ and $\vvt(\cdot, T) = 0$ and for all $\Omega$-periodic $\St \in \{C^\infty(\Omega\times (-\infty, T])^{2\times 2}\}$ 
with $\St(\cdot, T)= 0$ the following integral identities hold:
\begin{subequations}\label{weak}
\begin{align}
\begin{split}
 \intt \left\langle \pd{\vv}{t},\vvt \right\rangle \dt &+ \intt\intd (\vv\cdot \nabla)\vv\cdot \vvt \dx\dt= - \intt\intd \S:\D(\vvt)\dx\dt,   \label{weak1}
\end{split}\\
\begin{split}
 \intt \left\langle \pd{\S}{t},\St \right\rangle \dt & + \intt\intd (\vv\cdot\nabla)\S:\St\dx\dt + \intt\intd \S\W:\St - \W \S :\St \dx\dt \\ &= - \intt\left\langle \pd{\nabla \S}{t},\nabla\St \right\rangle\dt + \intt\intd \D:\St \dx\dt \label{weak2},
\end{split}
\end{align}
\end{subequations}
and, in addition, for all time independent,  $\Omega$-periodic and smooth $\vvt$ and $\St$  
$$
\lim_{t \to 0^+} \int_\Omega \vv(t) \cdot\vvt = \int_\Omega \vv_0\cdot \vvt, \qquad \lim_{t \to 0^+} \int_\Omega \S(t) :\St = \int_\Omega \S_0 :\St.
$$
\end{definition}
\begin{remark}
Observe that for \eqref{weak2} we need $ \pd{\nabla \S}{t} \in L^2(0,T;(W^{1,2})^\star)$. But the assumption  $\pd{ \S}{t} \in L^2(0,T;L^{2})$  automatically
implies it, in our simple setting of periodic functions with zero means.
\end{remark}
\begin{theorem}[Global-in-time existence and higher regularity]\label{th:1}
Let $\vv_0 \in W_{\div}^{1,2}$ and $\S_0 \in W^{2,2}$ be $\Omega$-periodic. Then, there exists a global in time weak solution $(\vv, \S)$ 
to the problem \eqref{sysr1}. Moreover, the initial condition is attained in the sense $\lim_{t\to 0}\|\S(t) - \S_0 \|_{W^{3/2,2}} = 0$, $\lim_{t\to 0}\|\vv(t) - \vv_0 \|_{L^2} =0 $ and the following higher regularity estimates hold:
\begin{equation}\begin{split}
\| \vv \|_{L^\infty(0,T;W^{1,2})}  &+ \| \S \|_{L^\infty(0,T;W^{2,2})} + \left\Vert \S \right\Vert_{W^{1,2} (0,T;W^{1,2})} \\
& + \left\Vert \pd{\S}{t} \right\Vert_{L^2 (0,T;W^{1,2})} + \left\Vert\pd{\vv}{t} \right\Vert_{L^2(0,T;(W^{1,2}_{div})^*)}  \le C\exp ( C \exp(T)),
\end{split}\label{M151}
\end{equation}
where $C = C  \left( \| \vv_0 \|_{W^{1,2}}  +\| \S_0   \|_{W^{2,2}} \right) $. Furthermore, the weak solution satisfying \eqref{M151} is uniquely determined by the initial data.\end{theorem}

\subsection[Regularization by $-\eps \Delta \S$. Local-in-time existence of a weak solution]{Regularization by $-\eps \Delta \S$: local-in-time or small data existence}

For the second regularization of the system \eqref{syso}, obtained by adding $-\eps\Delta \S$ to the left-hand side of \eqref{syso}$_3$, we will be able to prove, for fixed $\eps>0$, a weaker existence result: we either 
restrict ourselves to a short time interval or we establish a global in time existence result for small initial data. 

We investigate the following problem\footnote{Again, we set $\eps=1$ as taking the limit $\eps\to 0$ is beyond the scope of this paper.}: 
for given $\Omega$-periodic initial data $\vv_0, \S_0$, find $\vv(x,t):\Omega \times (0,T) \mapsto \Re^2$, \mbox{$p(x,t):\Omega \times (0,T) \mapsto \Re$} and $\S(x,t):\Omega \times (0,T) \mapsto \Re^{2\times 2}_{\text{sym}}$
that are $\Omega$-periodic and satisfy
\begin{subequations}\label{sys2}
\begin{align}
\div\vv &= 0, \label{sys2a}\\
\pd{\vv}{t} + (\vv\cdot\nabla)\vv   + \nabla p &= \div\S, \label{sys2b}\\
\pd{\S}{t} + (\vv \cdot \nabla)\S + \S\W -\W\!\S {- \Delta \S} &= \D\label{sys2c},\\
\vv(x,0) &= \vv_0,\\
\S(x,0) &= \S_0.
\end{align}
\end{subequations}

Let us now clarify what we mean by a {weak solution} to \eqref{sys2}.

\begin{definition}\label{def:dr}
A couple $(\vv,\S)$ satisfying   
$$\vv \in L^2 (0, T ; W^{1,2}_{div})\cap \mathcal{C}(0,T;L^2_{div}), \quad \S \in  L^2(0,T;W^{1,2})\cap \mathcal{C}(0,T;L^2),$$
$$ \pd{\vv}{t} \in L^2(0,T;(W_{\div}^{1,2})^\star), \quad \ \pd{\S}{t} \in L^2(0,T;(W^{1,2})^\star)$$
is called a 
weak solution to the problem \eqref{sys2} if, for all $\Omega$-periodic $ \vvt \in C^\infty(\Omega\times (-\infty, T])^2$ such 
that  $\div \vvt = 0$ and $\vvt(\cdot, T) = 0$ and for all $\Omega$-periodic $\St \in \{C^\infty(\Omega\times (-\infty, T])^{2\times 2}\}$ with 
$\St(\cdot, T)= 0$,  the following integral identities hold:
\begin{subequations}\label{weakth2}
\begin{align}
\begin{split}
 \intt \left\langle \pd{\vv}{t},\vvt \right\rangle \dt &+ \intt\intd (\vv\cdot \nabla)\vv\cdot \vvt \dx\dt= - \intt\intd \S:\D(\vvt)\dx\dt,   \label{weak1th2}
\end{split}\\
\begin{split}
 \intt \left\langle \pd{\S}{t},\St \right\rangle \dt & + \intt\intd (\vv\cdot\nabla)\S:\St\dx\dt + \intt\intd \S\W:\St - \W \S :\St \dx\dt \\ &+ \intt\intd \nabla\S:\nabla\St \dx\dt = \intt\intd \D:\St \dx\dt,\label{weak2th2}
\end{split}
\end{align}
\end{subequations} 
and for all time-independent,  $\Omega$-periodic and smooth $\vvt$ and $\St$ 
$$
\lim_{t \to 0^+} \int_\Omega \vv(t)\cdot \vvt = \int_\Omega \vv_0\cdot \vvt, \qquad \lim_{t \to 0^+} \int_\Omega \S(t) :\St = \int_\Omega \S_0 :\St.
$$
\end{definition}
The following existence result holds. 
\begin{theorem}[Local-in-time existence/small-data global-in-time existence]\label{th:2}
There exists a $c_*>0$ (a constant related to certain  interpolations and embeddings)  such that for any $\vv_0 \in W_{\div}^{1,2}$, $\S_0 \in W^{2,2}$ the problem \eqref{sys2} admits a weak solution $(\vv, \S)$ provided that for $X(0) := \| \nabla \vv_0 \|^2_{L^2} + \| \nabla \S_0 \|^2_{L^2}$ and $c_0:=c_*\max\{1, \|\S_0\|_2^2\}$ 
one of the following conditions holds:
\begin{itemize}
\item[(i)] $T <\frac{1}{c_0 X(0)}\;$; 
\item[(ii)] $ X (0) < c_0^{-\frac12}$.
\end{itemize}
The initial condition is attained in the sense $\lim_{t\to 0}\|\S(t) - \S_0 \|_{W^{1,2}}=0$ and $\lim_{t\to 0}\|\vv(t) - \vv_0 \|_{L^2}=0$. 
Moreover, this solution satisfies the following higher regularity estimates:
\be{ap21t}{
\| \vv    \|^2_{L^\infty (0, T; L^{2})} + \| S   \|^2_{L^\infty (0, T; L^{2})}   +\| \S   \|^2_{L^2 (0, T; W^{1,2})} \le   \|\vv_0 \|^2_{L^2} + \|\S_0 \|^2_{L^2},
}
\be{ap22t}{ \| \vv    \|^2_{L^\infty (0, T; W^{1,2})}  +\| \S   \|^2_{L^2 (0, T; W^{2,2})} \le C_1,
}
with $C_1 = CT\left(\frac{X (0)}{1 - c_0 T X (0)} \right)^2$ under hypothesis (i) and $C_1 = \frac{2 X(0)}{1 -  c_0 X (0)}$ under hypothesis (ii), as well as
\be{ap23t}{\left\Vert\pd{\vv}{t} \right\Vert_{L^2(0,T;(W^{1,2}_{div})^*)}  \le  C ((C_1 T)^\frac12 +1)  [\|\vv_0 \|_{L^2} + \|\S_0 \|_{L^2}],
}
\be{ap24t}{   \left\Vert\pd{\S}{t} \right\Vert^2_{L^2 (0,T;L^2)}  \le CC_1 (1+ C_1 +T).
}
Furthermore, the weak solution to \eqref{sys2} satisfying \eqref{ap21t}--\eqref{ap24t} is uniquely determined by the initial data.
\end{theorem}

\subsection{Logarithmic Sobolev inequality}

The logarithmic Sobolev inequality plays a crucial role in our analysis for the derivation of a-priori estimates. These kinds of critical Sobolev 
inequalities have been extensively studied in the context of the Euler equations, see for example Kozono \cite{Kozono2000, Kozono2002}.
The special case, that we use here, was given by Brezis and Gallouet \cite{Brezis1979} (see also Brezis and Wainger \cite{Brezis1980}), 
where the authors studied the nonlinear Schr\"odinger equation.  
\begin{prop}\label{proplogineq}

Let $f \in W^{2,2}(\Omega)$ and $ \Omega \subset \Re^2$ with boundary satisfying the strong local Lipschitz condition \cite[4.9]{adams}. 
Then the following Brezis--Gallouet inequality holds, see \cite{Brezis1979}:
\be{logineq}{ \| f \|_{L^\infty} \le C  \left( 1 +\| \grad f\|_{L^2}\left(\ln^+ (\|f \|_{W^{2,2}}) \right)^{\frac{1}{2}}\right),}
where $ \ln^+(x) = \begin{cases} 1 &\text{ for } x < e, \\ \ln x &\text{ for } x \ge e.
 \end{cases}$
\end{prop}
The inequality \eqref{logineq}, for complex valued functions, was proved in \cite{Brezis1979}. For the sake of completeness, 
we recall below the proof of \eqref{logineq}. For a bounded domain $\Omega \subset \Re^2$ satisfying the strong local Lipschitz 
condition, there is a bounded extension operator $E$ from $W^{2,2} (\Omega)$ to $W^{2,2}(\Re^2)$, see 
\cite[Chapter VI]{Stein1970}. Let us write $Ef = \ve{f}$ and let $\hat{\ve{f}}$ denote the Fourier transform of $\ve{f}$. We have
\begin{eqnarray}
\| (1+|\xi|)\hat{\ve{f}}\|_{L^2(\Re^2)} &\le & C\|\ve{f}\|_{W^{1,2}(\Re^2)},\\
\| (1+|\xi|^2)\hat{\ve{f}}\|_{L^2(\Re^2)} &\le & C\|\ve{f}\|_{W^{2,2}(\Re^2)},\\
\|\ve{f}\|_{L^\infty (\Re^2)} &\le & C \|\hat{\ve{f}}\|_{L^1(\Re^2)}. \label{inf1}
\end{eqnarray}

\begin{proof}(Proposition \ref{proplogineq})
\be{ppli}{
\begin{split}
\|\hat{\ve{f}}\|_{L^1(\Re^2)} &= \int_{|\xi|<R} |\hat{\ve{f}} |\mathrm{d\xi} + \int_{|\xi|\ge R} |\hat{\ve{f}} |\mathrm{d\xi}\\
=& \int_{|\xi|<R} (1+|\xi|) |\hat{\ve{f}} |  \frac{1}{1+|\xi|}\mathrm{d\xi} + \int_{|\xi|\ge R} (1+|\xi|^2) |\hat{\ve{f}} |  \frac{1}{1+|\xi|^2} \mathrm{d\xi}\\
\le& \left(\int_{|\xi|<R} (1+|\xi|)^2 |\hat{\ve{f}} |^2 \mathrm{d\xi} \right)^{\12}\left(\int_{|\xi|<R} \frac{1}{(1+|\xi|)^2} \mathrm{d\xi} \right)^{\12} \\
&+\left(\int_{|\xi|\ge R} (1+|\xi|^2)^2 |\hat{\ve{f}} |^2 \mathrm{d\xi} \right)^{\12}\left(\int_{|\xi|\ge R} \frac{1}{(1+|\xi|^2)^2} \mathrm{d\xi} \right)^{\12} \\
\le & C\|\ve{f}\|_{W^{1,2}(\Re^2)}\left(\ln(e + R) \right)^\12 + C\|\ve{f}\|_{W^{2,2}(\Re^2)}\frac{1}{1+R}.
\end{split}
}
Inequality \eqref{ppli} holds for every $R\ge 0$. We put $R = \|\ve{f}\|_{W^{2,2} (\Re^2)}$ and by \eqref{inf1} we get
$$\|\ve{f}\|_{L^\infty} \le C\left(1+\|\ve{f}\|_{W^{1,2}}\left(\ln^+(\|\ve{f}\|_{W^{2,2}(\Re^2)}) \right)^\12 \right).$$
Since $\ve{f}$ is a continuous extension of $f$, one obtains  \eqref{logineq}.
\end{proof}

\section{Proof of Theorem \ref{th:1} }\label{sec3}
\subsection{A priori estimates}\label{sec:apriori}
We first collect the a priori estimates related to the problem \eqref{sysr1}.

\begin{prop}\label{lem:apriori}
 Let $\vv_0 \in W^{1,2}_{\div}$ 
and $\S_0 \in W^{2,2}$. For sufficiently smooth  $\vv$ and $\S$ satisfying \eqref{sysr1} the following bounds~hold:
\begin{align*}\| \vv    \|_{L^\infty (0, T; W^{1,2})}  +\| \S   \|_{L^\infty (0, T; W^{2,2})} &\le 2 \exp \left[ C  \left( \| \vv_0    \|_{W^{1,2}} +\| \S_0   \|_{W^{2,2}} \right)  \exp (T) \right],\\
\left\Vert\pd{\vv}{t} \right\Vert_{L^2(0,T;(W^{1,2}_{div})^*)}  &\le  C T \exp \left[ C  \left( \| \vv_0    \|_{W^{1,2}} +\| \S_0   \|_{W^{2,2}} \right)  \exp (T) \right],\\
\intt  \left\Vert\pd{\S}{t} \right\Vert^2_{L^2}  +   \left\Vert\pd{\nabla\S}{t} \right\Vert^2_{L^2} \dt  &\le 
C( \|\vv_0\|_{L^2}) + C(\|S_0\|_{W^{1,2}})T + C T \left( 1 + C\left(\ln^+ (C T ) \right)^{\frac{1}{2}}\right).
\end{align*}
\end{prop}

\subsubsection{First a priori estimate}\label{1stae}

Taking the scalar product of \eqref{sysb} and $\vv$ and integrating the result over $\Omega$ we obtain 
\begin{align*} \intd \12 \pd{|\vv|^2}{t} \dx +  \intd \12 \vv \cdot \nabla |\vv|^2  \dx + \intd \nabla p\cdot \vv \dx =  \intd\div\S \cdot \vv \dx .
\end{align*}
Integrating the last three terms by parts, using the assumption of $\Omega$-periodicity and incorporating the divergence free 
condition \eqref{sysa}, we conclude that 
\be{step1v}{\12  \ddt \|\vv\|^2_{L^2} = - \intd \S: \nabla \vv\dx = - \intd \S:\D\dx,}
where we have also employed the symmetry of $\S$.

Next, taking the scalar product of  \eqref{sysc} and $\S$ and integrating the result over $\Omega$, we get 
\begin{align*}
\intd \12 \pd{|\S|^2}{t} \dx &+ \intd \12 \vv \cdot \nabla |\S|^2 \dx + \intd(\S\W -\W\!\S):\S \dx -  \intd \Delta \pd{\S}{t}: \S \dx =  \intd \D:\S \dx.
\end{align*}
Performing integrations by parts in the second and the fourth terms (using the periodicity of functions to eliminate the boundary integrals) and using 
\eqref{sysa}, we obtain 
\begin{align*} \12 \ddt\intd|\S|^2 \dx  & +  \intd(\S\W -\W\!\S):\S \dx +  \12 \ddt \intd |\nabla \S|^2 \dx =  \intd \D:\S \dx.
\end{align*}
Since, due to symmetry of $\S$ (see also \eqref{M100}),
\be{corot0}{
\begin{split}
(\S\W - \W\!\S) :\S & = 0,
\end{split}}
we conclude that
\be{step1S}{\12  \ddt \|\S \|^2_{L^2} + \12  \ddt \|\nabla \S \|^2_{L^2} = \intd \D:\S\dx.}
Taking the sum of \eqref{step1v} and \eqref{step1S}, noticing the mutual elimination of 
their right-hand sides and integrating the result over time, we finally arrive at
\be{ae1}{\|\vv(t) \|^2_{L^2} + \|\S(t) \|^2_{L^2} + \|\nabla\S(t)\|^2_{L^2} =  \|\vv_0 \|^2_{L^2} + \|\S_0 \|^2_{L^2} + \|\nabla \S_0 \|^2_{L^2} =: C_{(1)}. 
}

\subsubsection{Second a priori estimate} \label{312}
We take the scalar product of \eqref{sysb} and $- \Delta \vv$, integrate the result over $\Omega$, perform integration by parts and deduce, 
using again the periodicity of $\Omega$, that 
\begin{align}\label{M1}
\12 \ddt \|\nabla \vv \|^2_{L^2} + \intd \nabla((\vv\cdot\nabla)\vv) :\nabla \vv \dx 
&= \intd \nabla(\div\S) : \nabla \vv \dx.
\end{align}
Since $\div\vv = 0$ implies that $\pd{v_1}{x_1} = -\pd{v_2}{x_2}$, 
the term $\intd \nabla((\vv\cdot\nabla)\vv) :\nabla \vv \dx$ 
vanishes (see \cite{Temam2001} or \cite{MalekWMV} for details). As a consequence, we conclude from \eqref{M1} that  
\be{step2v}{ \12  \ddt \|\nabla \vv \|^2_{L^2} =  \intd \nabla(\div\S) : \nabla \vv\dx = \intd \Delta \S:\D \dx .
 }
Next, we take the scalar product of \eqref{sysb} and $\Delta \S$ and integrate over $\Omega$. We obtain
\begin{align*}
\intd\pd{\S}{t}:\Delta\S \dx &+ \intd(\vv \cdot \nabla)\S:\Delta\S \dx + \intd(\S\W -\W\!\S):\Delta\S \dx -  \intd \Delta \pd{\S}{t}: \Delta\S \dx =  \intd \D:\Delta\S \dx,
\end{align*}
which leads to 
\begin{align*}
 \12 \ddt\intd|\nabla \S|^2 \dx  &\!+\! \intd (\nabla\vv \cdot \nabla\S)\!:\!\nabla\S  \dx - \! \intd(\S\W \!-\W\!\S):\Delta \S \dx  + \12   \ddt \intd \!|\Delta \S|^2 \dx = - \!\intd \!\!\D\!:\!\Delta\S \dx, 
\end{align*}
where we have used the following identity (valid for $\vv$ fulfilling $\div\vv=0$):
\be{inertia2}{
\intd \nabla((\vv\cdot\nabla)\S):\nabla\S \dx= 
\intd \left(\pd{v_k}{x_l} \pd{S_{ij}}{x_k}\pd{S_{ij}}{x_l} + \12 \vv \cdot \nabla |\nabla \S|^2 \right)\dx=: \intd (\nabla\vv \cdot \nabla\S):\nabla\S  \dx.
}
Hence, we have
\be{step2S}{
\12  \ddt \left(\|\nabla\S \|^2_{L^2} + \|\Delta\S \|^2_{L^2}\right) + \intd(\nabla\vv \cdot \nabla\S):\nabla\S \dx + \intd (\W\!\S-\S\W):\Delta \S \dx = - \intd \D:\Delta\S \dx.
}
Summing \eqref{step2v} and \eqref{step2S} and taking advantage of the cancellation of their right-hand sides, we arrive at

\be{ae2}{
\begin{split}
  \ddt \left(\|\nabla \vv \|^2_{L^2} + \|\nabla\S \|^2_{L^2} + \|\Delta\S \|^2_{L^2}\right) = 2 \intd([\nabla\vv]\nabla\S):\nabla\S \dx - 2 \intd (\W\!\S-\S\W):\Delta \S \dx.
\end{split}
}
As there seems to be no cancellation concerning the terms on the right-hand side of \eqref{ae2}, the next step consists in estimating them. 
For the first term we apply the~embedding theorem, the Ladyzhenskaya interpolation inequality $\|z\|_4 \le  c \|z\|_2^{1/2}\|\nabla z\|_2^{1/2}$ and \eqref{ae1}:
\be{Iterm2}{
\begin{split}
\intd |(\nabla\vv \cdot \nabla\S):\nabla\S| \dx  &\le \|\nabla \vv \|_{L^2}\|\nabla\S \|_{L^4}^2\le 
c \|\nabla \vv \|_{L^2}\|\nabla\S\|_{2} \|\nabla^{(2)} \S \|_{2} \le C\|\nabla \vv \|_{L^2} \|\Delta \S\|_{2}\\ &\le C \left(\|\nabla \vv \|^2_{L^2} + \|\Delta \S \|^2_{L^2}\right).
\end{split} }
In order to treat the second term we use the logarithmic Sobolev inequality \eqref{logineq} in the following way:
\be{IIterm2}{
\begin{split}
\intd &(\W\!\S-\S\W):\Delta \S \dx \le 2\|\nabla \vv \|_{L^2} \|\S \|_{L^\infty}\|\Delta\S \|_{L^2}\\
& \le C\|\nabla \vv \|_{L^2}   \left( 1 +\| \grad \S\|_{L^2}\left(\ln^+ (\|\S \|_{W^{2,2}}) \right)^{\frac{1}{2}}\right)\|\Delta\S \|_{L^2}   \\
&\le C\|\nabla \vv \|_{L^2}\|\Delta\S \|_{L^2} + C \|\nabla \vv \|_{L^2}\| \grad \S\|_{L^2}\left(\ln^+ (\|\Delta\S \|_{L^2}) \right)^{\frac{1}{2}} \|\Delta\S \|_{L^2}\\
 & \le  C\left( \|\nabla \vv \|_{L^2}^2 + \|\Delta\S \|_{L^2}^2\right) + C  \|\nabla \vv \|_{L^2}\left(\ln^+ (\|\Delta\S \|_{L^2}) \right)^{\frac{1}{2}} \|\Delta\S \|_{L^2}\\
 & \le  C\left( \|\nabla \vv \|_{L^2}^2 + \|\Delta\S \|_{L^2}^2\right) + C\ln^+ (\|\Delta\S \|_{L^2}) \|\Delta\S \|_{L^2}^2.
\end{split} }
To summarize, using \eqref{ae2}, \eqref{Iterm2} and \eqref{IIterm2}, we deduce that
\be{New}{
\begin{split}
 \ddt \left( \| \nabla \vv \|^2_{L^2}  + \| \nabla \S \|^2_{L^2} + \| \Delta \S \|^2_{L^2} \right) &\le C\left( \|\nabla \vv \|_{L^2}^2 + \|\Delta\S \|_{L^2}^2\right) + C\ln^+ (\|\Delta\S \|_{L^2}) \|\Delta\S \|_{L^2}^2.
\end{split}}
For simplicity, we increase the right-hand side of \eqref{New} by adding some positive terms and taking advantage of the fact that $\ln^+(x) \ge 1$ and obtain 
\be{New2}{
\begin{split}
 &\ddt \left( \| \nabla \vv \|^2_{L^2}  + \| \nabla \S \|^2_{L^2} + \| \Delta \S \|^2_{L^2} \right) \\
 & \le C\left( \| \nabla \vv \|^2_{L^2}  + \| \nabla \S \|^2_{L^2} + \| \Delta \S \|^2_{L^2} \right)\left( 1+ \ln^+ \left( \| \nabla \vv \|^2_{L^2}  + \| \nabla \S \|^2_{L^2} + \| \Delta \S \|^2_{L^2} \right)\right).
\end{split}}
Let us denote $Y =  \| \nabla \vv \|^2_{L^2}  + \| \nabla \S \|^2_{L^2} + \| \Delta \S \|^2_{L^2}$ and rewrite \eqref{New2} as
\be{odeineq}{ \ddt Y \le Y + Y\ln^+ Y \le 2 Y\ln^+ Y \implies \ddt Y_e \le 2 Y_e \ln Y_e,}
where $Y_e =: e + Y$ ($\ln e =1$). Consequently,
$$ \ddt \ln(Y_e) \le 2 \ln Y_e \quad \implies \quad Y_e(t) \le \exp \left[\ln (Y_e(0)) \exp (2t) \right]. $$
Recalling the definitions of $Y$ and $Y_e$, the last inequality implies that, for all $t\in [0,T]$, 
\be{odefinA}
{\left( \| \nabla \vv    \|^2_{L^2}  + \| \nabla \S    \|^2_{L^2} + \| \Delta \S   \|^2_{L^2}\right)(t)  \le \exp \left[\ln (e +  \left( \| \nabla \vv_0  \|^2_{L^2}  + \| \nabla \S_0  \|^2_{L^2} + \| \Delta \S_0  \|^2_{L^2} \right))  \exp (2T)\right]. 
}
Finally, \eqref{odefinA} and \eqref{ae1} imply that, for all $t\in [0,T]$,
\be{odefin}
{\left( \| \vv    \|_{W^{1,2}}  +\| \S   \|_{W^{2,2}}\right)(t)  \le 2 \exp \left[ C  \left( \| \vv_0 \|_{W^{1,2}}  +\| \S_0   \|_{W^{2,2}} \right)  \exp (T) \right]  \equiv C_{(2)}. 
}

\subsubsection{A priori estimates for the time derivative of $\vv$ and $\S$}\label{313}

In order to gain compactness for $\vv$, $\S$ and $\nabla\S$, we  estimate their time derivatives.  First, note that 
(for brevity, the space $L^2(0,T; W^{1,2}_{div})$ is denoted by $X$ in the following lines)
\be{timeAp}{
\begin{aligned}
  \left\Vert\pd{\vv}{t} \right\Vert_{X^*}  &= \sup_{ \|\vvt \|_{X} \le 1} \left\vert \intt\intd \pd{\vv}{t}\cdot \vvt \dx\dt \right\vert \\
  &= \sup_{ \|\vvt \|_{X} \le 1} \intt \intd |(\vv\otimes \vv):\nabla\vvt| + |\S: \nabla \vvt| \dx\dt \\
  &\le\sup_{ \|\vvt \|_{X} \le 1} \left[ \intt \|\vv\|_{L^4}^2\|\nabla \vvt\|_{L^2}\dt +\intt \|\S\|_{L^2}\|\nabla \vvt\|_{L^2}\dt\right]\\
  &\le\sup_{ \|\vvt \|_{X} \le 1} C\left[ \intt \|\vv\|_{L^2} \|\nabla\vv\|_{L^2} \|\nabla \vvt\|_{L^2}\dt \right.  + \left. \intt \|\S\|_{L^2}\|\nabla \vvt\|_{L^2}\dt\right]\\
 &\le C T^\12 \left(\|\vv\|_{L^\infty(0,T;L^2)}  \|\nabla \vv\|_{L^\infty(0,T;L^2)} + 
\|\S\|_{L^\infty(0,T;L^2)}\right). 
\end{aligned}
}
Consequently, with help of \eqref{ae1} and \eqref{odefin}, we obtain
\be{timeA}{
  \left\Vert\pd{\vv}{t} \right\Vert_{L^2(0,T;(W^{1,2}_{div})^*)}  \le  C T^\12 C_{(1)} (1+ C_{(2)}).
}
In order to estimate $\pd{\S}{t}$, we take the scalar product of \eqref{sysc} and $\pd{\S}{t}$ and 
integrate the result over $(0,T) \times \Omega$. This leads to 
\begin{align*}
\intt\intd\pd{\S}{t}:\pd{\S}{t} \dx\dt &+ \intt\intd(\vv \cdot \nabla)\S:\pd{\S}{t} \dx\dt + \intt\intd(\S\W -\W\!\S):\pd{\S}{t} \dx\dt \\&- \intt\intd \Delta \pd{\S}{t}: \pd{\S}{t} \dx\dt = \intt \intd \D:\pd{\S}{t} \dx\dt.
\end{align*}
Hence
\be{testi1}{
 \intt  \left\Vert\pd{\S}{t} \right\Vert^2_{L^2}  +   \left\Vert\pd{\nabla\S}{t} \right\Vert^2_{L^2} \dt  \le \intt\intd  g \left\vert \pd{\S}{t} \right\vert \dx\dt,
}
where
$$
g:= \left\vert(\vv \cdot \nabla)\S\right\vert  +  |\S\W -\W\!\S| +  |\D|.
$$
Since, using \eqref{ae1} and \eqref{odefin}, 
\be{testi2}{
\begin{split} \intd \left\vert(\vv \cdot \nabla)\S\right\vert ^2 \dx &\le  \intd |\vv|^2 |\nabla \S|^2 \dx \le  \|v\|_{L^4}^2 \|\nabla\S\|_{L^4}^2 \\ &\le C \|v\|_{L^2} \|\nabla\vv\|_{L^2} \|\S\|_{L^2}\|\nabla\S\|_{L^2} \le C_{(1)}^2 C_{(2)}^2.\end{split} 
}
and further, with help of the logarithmic Sobolev inequality \eqref{logineq}
\be{testi4}{
\begin{split}
\intd |\S\W  - \W\!\S|^2 \dx &\le 2 \intd |\S|^2 |\nabla \vv|^2 \dx \le \|\S\|_{L^{\infty}}^2 \|\nabla \vv|_{L^2}^2 \\
&\le C  \left( 1 +\| \grad    \S\|_{L^2}^2 \ln^+ (\|\S \|_{W^{2,2}}) \right) \|\nabla \vv|_{L^2}^2  < +\infty 
\end{split}}
due to \eqref{odefin} and finally
\be{testi3}{ 
\intd |\D|^2 \dx \le \|\vv\|_{W^{1,2}}^2 \le C_{(2)}^2,}
we conclude that $g$ is bounded in $L^2(\Omega)$ uniformly w.r.t. time $t\in [0,T]$. Consequently $g$ is bounded uniformly in $L^2(0,T; L^2(\Omega))$ 
and it then follows from \eqref{testi1}, using Young's inequality, that  \be{st}
{ \intt  \left\Vert\pd{\S}{t} \right\Vert^2_{L^2}  +   \left\Vert\pd{\nabla\S}{t} \right\Vert^2_{L^2} \dt  \le C_{(3)}. }
Referring to \eqref{ae1}, \eqref{odefin}, \eqref{timeA}, and \eqref{st} we observe  that the assertions of Proposition \ref{lem:apriori} are thus proved.

\subsection{Galerkin approximation}

We prove Theorem \ref{th:1} by means of a Galerkin approximation. Following a standard procedure we employ orthonormal countable bases, 
$(\omega^i)_{i=1}^\infty$  and $(w^i)_{i=1}^\infty$, of the spaces  $W^{1,2}_{div}$ and $W^{1,2}$, respectively. 

Let $N\ge 1$ be fixed. The functions $\vv^N(x,t) = \sum^N_{i=1}c_i^N(t)\omega^i(x)$ and $\S^N(x,t) = \sum^N_{i=1}d_i^N(t)w^i(x)$ 
are called the $N$th Galerkin approximation, if  $(c_1^N , \dots , c_N^N,$ $d_1^N , \dots , d_N^N)$ solve the system of ordinary differential equations
\begin{subequations}\label{gal}
\begin{align}
 \intd \pd{\vv^N}{t} \cdot\omega^i + \intd ((\vv^N \cdot \nabla) \vv^N)\cdot \omega^i & = - \intd \S^N : \nabla \omega^i  \quad \textrm{  for all } i = 1,\dots,N, \\ 
\begin{split}
 \intd \pd{\S^N}{t}:w^j &+ \intd (\vv^N\cdot \nabla)\S^N:w^j + \intd (\S^N\W^N-\W^N\S^N):w^j \\
&+ \intd \nabla \pd{\S^N}{t}: \nabla w^j =  \intd \D^N : w^j \quad \textrm{  for all } j = 1,\dots,N,
\end{split}\\
&c_i^N(0) =\intd \vv_0 \cdot \omega^i, \quad d_j^N(0) =\intd \S_0:w^j, \quad 1\le i,j\le N,
\end{align}
\end{subequations}
with the initial conditions $ \vv^N(x,0) = P_v^N\vv_0(x), \quad \S^N(x,0) = P_S^N \S_0(x),$ 
where $P_v^N$ and $P_S^N$ are proper orthogonal continuous projections.

The existence of continuous functions $(c_1^N , \dots , c_N^N,$ $d_1^N , \dots , d_N^N)$, that solve \eqref{gal}, 
follows from the classical Carath\'eodory theorem. The uniform estimates, that we state in the next section, enable 
us to extend the solution onto the whole time interval $[0, T]$.

\subsection{Limit $N \to \infty$}\label{sec:limpas}

Recalling the energy estimates from Section \ref{sec:apriori},  it is not difficult to see that the 
$N$th Galerkin approximation satisfies, for $N$ arbitrary, the following  estimates:
 $$ \| \vv^N    \|_{L^\infty (0, T; W^{1,2})}  +\| \S^N   \|_{L^\infty (0, T; W^{2,2})} \le C (\Omega , T,\vv_0,\S_0 ) $$
$$ \left\Vert\pd{\vv^N}{t} \right\Vert_{L^2(0,T;(W^{1,2}_{div})^*)}  \le C (\Omega , T,\vv_0,\S_0 ),$$
$$\intt  \left\Vert\pd{\S^N}{t} \right\Vert^2_{L^2}  +   \left\Vert\pd{\nabla\S^N}{t} \right\Vert^2_{L^2} \dt  \le 
C (\Omega , T,\vv_0,\S_0 ).$$
Thanks to the above estimates that are uniform with respect to $N$, sequential weak or *-weak precompactness 
of the function spaces involved, and thanks to the identification of the time derivative of a limit function 
with the limit of the time derivative via the distributional formula for the time derivative, we observe that
for a selected (not relabelled) subsequence we have 
\begin{align}\begin{aligned}
\vv^N &\rightharpoonup^*  \, {\vv}\quad &&\text{ in } L^\infty(0,T;W^{1,2}),\\
\pd{\vv^N}{t} &\rightharpoonup  \,\,\,\pd{\vv}{t} \quad &&\text{ in } L^2(0,T;(W^{1,2})^*), \\
\S^N &\rightharpoonup^*  \,{\S}\quad  &&\text{ in } L^\infty(0,T; W^{2,2}),\\
\pd{\S^N}{t} &\rightharpoonup   \,\,\,\pd{\S}{t} \quad && \text{ in } L^2(0,T; (W^{1,2})).
\end{aligned}\end{align}
Weak convergence suffices to take the limit in the linear terms in \eqref{gal}. Moreover, 
since $$ W^{1,2} \hookrightarrow\hookrightarrow L^4 \hookrightarrow L^2 = (L^2)^* \hookrightarrow (W^{1,2})^*$$ 
we get, thanks to the Aubin--Lions compactness lemma, see for example \cite{Simon1987}, the following strong convergence results
$$\S^N  \longrightarrow  {\S}\quad \text{ in } L^2(0,T;L^4),$$
$$\vv^N  \longrightarrow  {\vv}\quad \text{ in } L^2(0,T;L^4).$$
This allows us to take limit in the nonlinear terms
\be{nonpas1}{\intt \intd ((\S\W -\W\!\S) - (\S^N\W^N-\W^N\S^N)):\St \dx\dt.}
To illustrate this, let us consider one term of \eqref{nonpas1}:
\begin{align*}
\intt &\intd (\S\W - \S^N\W^N):\St \dx\dt  =  \intt \intd (\S\W^N-\S^N\W^N - \S\W^N + \S\W):\St \dx\dt \\
&=\intt \intd (\S - \S^N)\W^N:\St\dx\dt + \intt\intd\S(\W-\W^N):\St \dx\dt\\
& \le \|\S - \S^N\|_{L^2(L^4)} \|\W\|_{L^2(L^2)}\|\St\|_{L^\infty(L^4)}  + \intt\intd (\W-\W^N):\S\St \dx\dt  \longrightarrow 0.
\end{align*}
A standard (similar) approach is used in order to take the limit in the convective terms. Consequently, we can conclude that 
$\vv$, $\S$ satisfy \eqref{weak1}, \eqref{weak2}.
Moreover, the uniform estimates mentioned above and the weak lower semicontinuity of respective norms imply the  estimates 
for the functions $\vv$ and $\S$ as stated in Theorem \ref{th:1}. In addition, thanks to standard space-time interpolation 
of ${\vv} \text{ in } L^2(0,T;W^{1,2})$  and $\pd{\vv}{t} \text{ in } L^2(0,T;(W^{1,2})^*)$ we obtain $\vv \in \mathcal{C}([0,T]; L^2)$ 
(see for example \cite{Simon1987}). 
Similarly, $\S \in \mathcal{C}([0,T]; W^{3/2,2})$. Hence 
\be{tlim}{ \lim_{t \to 0^+} \|\vv(t) - \vv(0)\|_{L^2} \quad \textrm{ and } \quad  \lim_{t \to 0^+} \| \S(t) - \S(0) \|_{W^{3/2,2}}.}

To verify the statements of Theorem \ref{th:1} regarding the initial conditions, it thus remains to check that $\vv(0) = \vv_0$ and $\S (0) = \S_0$. 
For this purpose, we multiply both  equations in \eqref{gal} by  $\psi \in C^\infty ((-\infty,T])$ satisfying $\psi(T) = 0$ and 
integrate over the time interval $(0,T)$. Then, integration by parts with respect to time leads to ($i=1,\dots,N$), 
\begin{subequations}\label{time0}
\begin{align}
 -&\intt\!\!\intd \vv^N \cdot \omega^i \pd{\psi}{t} -\intt \!\!\intd \vv^N(0) \cdot \omega^i \psi(0) + \intt\!\!\intd ((\vv^N \cdot \nabla) \vv^N)\cdot \omega^i\psi  = - \intt\!\!\intd \S^N : \nabla \omega^i\psi,   \\ 
 \begin{split}
-&\intt \!\!\intd \S^N:w^j \pd{\psi}{t} -\intt \!\!\intd \S^N(0): w^j \psi(0)+ \intt\!\!\intd (\vv^N\cdot\nabla)\S^N:w^j \psi -\intt \!\!\intd \nabla \S^N:\nabla w^j \pd{\psi}{t}\\
 &\qquad -\intt \!\!\intd \nabla\S^N(0) : \nabla w^j \psi(0) + \intt\!\!\intd (\S^N\W^N-\W^N\S^N):w^j \psi =  \intt\!\!\intd \D^N : w^j \psi 
\end{split}
\end{align}
\end{subequations}
Next, letting $N \to\infty$ and referring to the completeness of $(\omega^i )_{i=1}^\infty$ and $(w^i )_{i=1}^\infty$ in $W^{1,2}_{div}$ and $W^{1,2}$, 
respectively, we get, for smooth spatial test functions $\vvt$ and $\St$, the following identities:
\begin{subequations}\label{time}
\begin{align}
 -\intt \!\! \intd \vv \cdot\vvt \pd{\psi}{t} &-\intt \!\! \intd \vv_0\cdot \vvt \psi(0) + \intt \!\!\intd ((\vv \cdot \nabla) \vv)\cdot\vvt\psi  = - \intt \!\!\intd \S : \nabla \vvt\psi,   \\ 
\begin{split}
-\intt \!\! \intd \S : \St \pd{\psi}{t} &-\intt \!\! \intd \S_0 :\St \psi(0)+ \intt \!\!\intd (\vv\cdot\nabla)\S:\St \psi + \intt \!\!\intd (\S\W-\W\!\S):\St \psi\\
 & -\intt \!\! \intd \nabla \S : \nabla \St \pd{\psi}{t} -\intt \!\! \intd \nabla\S_0 :\nabla \St \psi(0) =  \intt \!\!\intd \D : \St \psi. 
\end{split}
\end{align}
\end{subequations}
Since by properties of a generalized derivative one has 
\begin{align}
\intt \left\langle \pd{\vv}{t},\vvt \right\rangle\psi \dt &= -\intt \!\! \intd \vv\cdot \vvt \pd{\psi}{t} -\intt \!\! \intd \vv(0) \cdot\vvt \psi(0),\label{wdp1}\\
\intt \left\langle \pd{\S}{t},\St \right\rangle \dt &= -\intt \!\! \intd \S: \St \pd{\psi}{t} -\intt \!\! \intd \S(0) :\St \psi(0),\label{wdp2}
\end{align}

comparing \eqref{weak1}, \eqref{weak2} with \eqref{time} 
and choosing $\psi(0) \neq 0$ we obtain 
$$ \intd \vv(0)\cdot\vvt \dx = \intd \vv_0\cdot\vvt\dx,\quad \intd \S(0):\St \dx = \intd \S_0:\St\dx. $$
Hence $\vv(0) = \vv_0$, $\S (0) = \S_0$ a.e. and by virtue of  \eqref{tlim} we have  
$$\lim_{t \to 0^+} \|\vv(t) -\vv(0)\|_{L^2} \quad \lim_{t \to 0^+} \| \S(t) - \S(0) \|_{W^{3/2,2}}.$$
We have thus proved that $(\vv,\S)$ is a  weak solution to \eqref{sysr1}. The proof of uniqueness of the weak solution fulfilling 
the established regularity results is standard. The proof of Theorem \ref{th:1} is complete. 

\section{Proof of Theorem \ref{th:2} }\label{sec4}

\subsection{A priori estimates}\label{sec:R2apriori}

\begin{prop}\label{lem:apriori2}
Let $\vv_0 \in W^{1,2}_{\div}$, $\S_0 \in W^{2,2}$ and $X(0) := \| \nabla \vv_0 \|^2_{L^2} + \| \nabla \S_0 \|^2_{L^2}$. 
For sufficiently smooth  $\vv$ and $\S$ satisfying \eqref{sys2} assume that either $T <\frac{1}{c_0 X(0)}$ or $ 1 > c_0 X(0)$, 
where $c_0:= c_* \max\{1, \|\S_0\|_2^2\}$ and $c_*$ is an absolute constant related to certain interpolations and embeddings. Then
\be{ap21}{
\| \vv    \|^2_{L^\infty (0, T; L^{2})} + \| S   \|^2_{L^\infty (0, T; L^{2})}   +\| \S   \|^2_{L^2 (0, T; W^{1,2})} \le   \|\vv_0 \|^2_{L^2} + \|\S_0 \|^2_{L^2},
}
\be{ap22}{ \| \vv    \|^2_{L^\infty (0, T; W^{1,2})}  +\| \S   \|^2_{L^2 (0, T; W^{2,2})} \le C_1,
}
where $C_1 = CT\left(\frac{X (0)}{1 - c_0 T X (0)} \right)^2$ under the hypothesis $T <\frac{1}{c_0X(0)}$ and $C_1 = \frac{2 X(0)}{1 -  c_0 X (0)}$ under the hypothesis $1 > c_0X(0)$. Moreover,
\be{ap23}{\left\Vert\pd{\vv}{t} \right\Vert_{L^2(0,T;(W^{1,2}_{div})^*)}  \le  C ((C_1 T)^\frac12 +1)  [\|\vv_0 \|_{L^2} + \|\S_0 \|_{L^2}],
}
\be{ap24}{ \intt  \left\Vert\pd{\S}{t} \right\Vert^2_{L^2} \dt  \le C C_1 (1+ T + C_1).
}
\end{prop}

\begin{proof}
The first energy estimate \eqref{ap21} is arrived at along the lines giving \eqref{ae1} for \eqref{sysr1}.
The second energy estimate yields (see Subsect. \ref{312} for details) 
\be{step2va}{\12  \ddt \|\nabla \vv \|^2_{L^2} =  \intd \Delta\S : \D \dx,}
\be{step2Sa}{
\begin{split}
\12 \ddt \|\nabla\S \|^2_{L^2}  + \|\Delta\S \|^2_{L^2} =  &- \intd([\nabla\vv]\nabla\S):\nabla\S \dx + \intd (\W\!\S-\S\W):\Delta \S \dx  -\intd\D:\Delta\S \dx.
\end{split}}
The convective term is estimated by means of interpolation and Young's inequalities as follows:
\be{Iterm}{
\begin{split}
\intd|(\nabla\vv \cdot \nabla\S):\nabla\S| \dx  &\le \|\nabla \vv \|_{L^2}\|\nabla\S \|_{L^4}^2 \le c\|\nabla \vv \|_{L^2}\|\nabla\S \|_{L^2} \|\Delta \S\|_{L^2} \\&\le c \|\nabla \vv \|^2_{L^2}\|\nabla\S \|^2_{L^2} + \frac14 \|\Delta \S \|^2_{L^2}.
\end{split} }
To estimate the second term on the right-hand side of  \eqref{step2Sa} we employ Agmon's inequality (in 2d) 
$\|z\|_{L^{\infty}} \le c\|z\|_{L^2}^{1/2} \|z\|_{W^{2,2}}^{1/2}$ and obtain (using also \eqref{ap21})
\be{IIterm}{
\begin{split}
\intd &(\W\!\S-\S\W):\Delta \S \dx \le 2\|\nabla \vv \|_{L^2} \|\S \|_{L^\infty}\|\Delta\S \|_{L^2} \le c\|\nabla \vv \|_{L^2} \| \S \|_{L^2}^\12\| \S \|_{W^{2,2}}^\12 \|\Delta\S \|_{L^2}   \\
 &\le c\|\nabla \vv \|_{L^2}\| \S \|_{L^2}^\12 \|\Delta\S \|_{L^2}^{\frac{3}{2}}  \le \frac{c_*}{2}\|\S_0 \|^2_{L^2} \|\nabla \vv \|^4_{L^2}+ \frac14 \|\Delta \S \|^2_{L^2}. 
\end{split} }
Summing up  \eqref{step2va} and \eqref{step2Sa}, using \eqref{Iterm} and \eqref{IIterm}, and setting $c_0:= c_* \max\{1, \|\S_0\|^2_2\}$, we conclude that 
\be{2ae1}{ \ddt \left( \| \nabla \vv \|^2_{L^2} + \| \nabla \S \|^2_{L^2} \right) + \| \Delta \S \|^2_{L^2}  \le  c_* \left(\| \S \|^2_{L^2} \|\nabla \vv \|^4_{L^2}+ \|\nabla \vv \|^2_{L^2} \|\nabla\S \|^2_{L^2}\right) \le c_0\left(\|\nabla \vv \|^2_{L^2} + \|\nabla\S \|^2_{L^2}\right)^2}
since $\| S \|_{L^2} \le \| S_0 \|_{L^2} $.
If we set $X := \left( \| \nabla \vv \|^2_{L^2} + \| \nabla \S \|^2_{L^2} \right) $ and $Y := \| \Delta \S \|^2_{L^2}$, then \eqref{2ae1} takes 
the form $\ddt X + Y \le c_0 X^2$. Since $$\ddt X  \le c_0 X^2 \Longrightarrow  X(t) \le \frac{X (0)}{1 - c_0 t X (0)} \quad \text{ and } 
\quad X(T) + \intt Y \le CT\left(\frac{X (0)}{1 -  c_0 T X (0)} \right)^2,$$
we conclude that, for $T <\frac{1}{c_0X(0)}$,
$$\vv \in L^\infty(0,T;W^{1,2}), \quad \S \in L^\infty(0,T;W^{1,2}) \cap L^2(0,T;W^{2,2}),$$
which implies the first desired estimate under the assumption on the smallness of the time interval, i.e.  $T <\frac{1}{c_0X(0)}$.

Let us now justify the same estimate under the assumption that the initial data are small in the relevant norms. 
Starting from the first inequality given in   \eqref{2ae1} and using the notation introduced above and  estimating  
$\| \nabla \S \|^2_{L^2}$ by $cY$, we obtain $\ddt X + Y \le c_0X^2 Y$, which leads to 
\be{ode1}{\ddt X + (1- c_0X^2) Y \le 0 .}

Assuming that $1 > \delta > c_0X^2(0) = \left( \| \nabla \vv_0 \|^2_{L^2} + \| \nabla \S_0 \|^2_{L^2} \right)^2$ and that $X(t)$ is continuous in time, we set 
\[
t^* := \inf \{t >0: c_0X^2 (t) = \delta \}.
\]
We will argue that $t^*$ cannot be finite. Indeed, assuming that 
$t^* \in (0, \infty)$ and integrating \eqref{ode1} over $(0,t^*)$ (noting that $t^*$ is the first possible time where $c_0X^2 (\cdot) = \delta$), 
we conclude that $ X(t^*) \le  X(0)$.
Consequently, $\delta= c_0X^2(t^*) \le  c_0X^2(0) <\delta$, which contradicts to the fact that $t^*$ is finite.  
Hence, using again \eqref{ode1}, we observe that, for any finite time $t$,
\[
X(t) + (1-\delta) \int_0^t Y \le X(0),
\]
which gives \eqref{ap22}. The value of the constant $C_1$ follows from the choice of $\delta := \frac{1+ c_0 X^2(0)}{2}$.

It remains to estimate the time derivatives of $\vv$ and $\S$. Since the systems \eqref{sysr1} and \eqref{sys2} differ only in the stress evolution 
equation, the estimation \eqref{ap23} of $\pd{\vv}{t}$ follows from \eqref{timeAp} and \eqref{ap22}.

In order to estimate $\pd{\S}{t}$, we take the scalar product of \eqref{sys2c} and $\pd{\S}{t}$ and integrate the result over $(0,T)\times\Omega$. This yields
\begin{align*}
\intt \!\!\intd\pd{\S}{t}:\pd{\S}{t} \dx\dt &+ \intt \!\!\intd(\vv \cdot \nabla)\S:\pd{\S}{t} \dx\dt + \intt \!\!\intd(\S\W -\W\!\S):\pd{\S}{t} \dx\dt \\&- \intt \!\!\intd \Delta \S: \pd{\S}{t} \dx\dt = \intt \!\! \intd \D:\pd{\S}{t} \dx\dt.
\end{align*}
Hence
\be{test2i1}{
 \intt  \left\Vert\pd{\S}{t} \right\Vert^2_{L^2} \dt  \le \intt \!\!\intd  g\left| \pd{\S}{t}\right| \dx\dt ,
}
where
$$
g:= \left\vert(\vv \cdot \nabla)\S\right\vert  +  |\S\W -\W\!\S| + |\Delta \S| +|\D|.
$$
Proceeding similarly as in Subsect. \ref{313}, it is possible to show that $g$ is bounded in $L^2(0,T; L^2)$. 
Young's inequality thus implies  \eqref{ap24}.

\end{proof}

\subsection{Galerkin approximation}

We proceed similarly as in the proof of Theorem \ref{th:1} and apply the same Galerkin scheme. 
Here, the functions $\vv^N(x,t) = \sum^N_{i=1}c_i^N(t)\omega^i(x)$ and $\S^N(x,t) = \sum^N_{i=1}d_i^N(t)w^i(x)$ 
solve the following  system of ordinary differential equations:
\begin{subequations}\label{gal2}
\begin{align}
 \intd \pd{\vv^N}{t} \cdot \omega^i + \intd (\vv^N \cdot \nabla) \vv^N)\cdot\omega^i & = - \intd \S^N : \nabla \omega^i  \quad \quad\textrm{ for all } i = 1,\dots,N, \\ 
 \begin{split}
 \intd \pd{\S^N}{t} : w^j + \intd (\vv^N\cdot\nabla)\S^N:w^j &+ \intd (\S^N\W^N-\W^N\S^N):w^j \\
+ \intd \nabla \S^N: \nabla w^j &=  \intd \D^N : w^j \quad \textrm{ for all } j = 1,\dots,N,
\end{split}\\
c_i^N(0) &=\intd \vv_0 \omega^i,\quad d_j^N(0) =\intd \S_0 w^j, \quad 1\le i,j\le N.
\end{align}
\end{subequations}

The existence of continuous functions $(c_1^N , \dots , c_N^N,$ $d_1^N , \dots , d_N^N)$, that solve \eqref{gal2}, 
follows from the classical Caratheodory theorem.

\subsection{Uniform estimates and Limit $N\to \infty$}

At this juncture, we can ``repeat'' the a priori estimates established for the sufficiently regular solution discussed in Subsect. \ref{sec:R2apriori}. 
Similarly as in the proof of Theorem \ref{th:1} we conclude corresponding weak and *-weak convergences for selected but not relabelled subsequences 
of $\{\vv^N\}$ and $\{\S^N\}$. 
Repeating the arguments from Section \ref{sec:limpas} we also obtain the following strong convergence results:
$$\vv^N  \longrightarrow  {\vv}\quad \text{ in } L^2(0,T;L^4),$$
$$\S^N  \longrightarrow  {\S}\quad \text{ in } L^2(0,T;L^4).$$
Consequently, we can take the limit in all terms of the governing equations and conclude that $\vv$ and $\S$ satisfy \eqref{weak1th2} and \eqref{weak2th2}.

Moreover, thanks to the space-time interpolation of ${\vv} \text{ in } L^2(0,T;W^{1,2})$  and $\pd{\vv}{t} \text{ in } L^2(0,T;(W^{1,2})^*)$ 
we obtain $\vv \in \mathcal{C}([0,T]; L^2)$. Similarly, $\S \in \mathcal{C}([0,T]; W^{1,2})$. Hence 
\be{tlimth2}{ \lim_{t \to 0^+} \|\vv(t) - \vv(0)\|_{L^2} \quad \lim_{t \to 0^+} \| \S(t) - \S(0) \|_{W^{1,2}}.}

The procedure of checking the attainment of the initial conditions proceeds in a similar way as in Section \ref{sec:limpas}: 
we have to show that $\vv(0) = \vv_0$ and $\S(0) = \S_0$. To prove these, we multiply \eqref{gal2} by  $\psi \in C^\infty ([0,T])$ 
satisfying $\psi(T) = 0$ and integrate over $(0,t)$. Then we integrate by parts w.r.t. time and take the limit $N \to\infty$ as in \eqref{time0}. Due to the completeness of the bases $(\omega^i )_{i=1}^\infty$ and $(w^i )_{i=1}^\infty$ in $W^{1,2}_{div}$ and $W^{1,2}$, we get for smooth spatial test functions $\vvt$ and $\St$ the following identities:
\begin{subequations}\label{timeth2}
\begin{align}
 -\intt \!\! \intd \vv \vvt \cdot \pd{\psi}{t} &-\intt \!\! \intd \vv_0 \cdot\vvt \psi(0) + \intt \!\!\intd ((\vv \cdot \nabla) \vv)\cdot\vvt\psi  = - \intt \!\!\intd \S : \nabla \vvt\psi ,  \\ 
\begin{split}
-\intt \!\! \intd \S: \St \pd{\psi}{t} &-\intt \!\! \intd \S_0: \St \psi(0)+ \intt \!\!\intd (\vv\cdot\nabla)\S:\St \psi \\
 & + \intt \!\!\intd (\S\W-\W\!\S):\St \psi +\intt \!\! \intd \nabla \S :\nabla \St \psi =  \intt \!\!\intd \D : \St \psi .
\end{split}
\end{align}
\end{subequations}
Using the relations \eqref{wdp1} and \eqref{wdp2} and comparing \eqref{weak1th2} and \eqref{weak2th2} with \eqref{timeth2}, 
we conclude (by~choosing $\psi(0) \neq 0$) that 
$$ \intd \vv(0)\cdot\vvt \dx = \intd \vv_0\cdot\vvt\dx,\quad \intd \S(0):\St \dx = \intd \S_0:\St\dx. $$
Hence $\vv(0) = \vv_0$, $\S (0) = \S_0$ a.e. and owing to \eqref{tlimth2} we observe that 
$$ \lim_{t \to 0^+} \|\vv(t) - \vv(0)\|_{L^2}, \quad \lim_{t \to 0^+} \| \S(t) - \S(0) \|_{W^{1,2}}.$$
We have proved that $(\vv,\S)$ is a  weak solution to \eqref{sys2} and have completed the proof of Theorem \ref{th:2}, since the proof of uniqueness of the weak solution fulfilling 
the established regularity results is again standard.

\section*{Appendix - Main difficulty in proving well-posedness of \eqref{syso} }\label{app:md}

Let us discuss the fundamental difficulties that obstruct the proof of local-in-time well-posedness of the system \eqref{syso}. 
We are interested in solving the spatially periodic problem. We can treat the equation \eqref{sysob} as an Euler equation with the right-hand side $\div\S$. 
Therefore, we apply the strategy that was used to show the local-in-time existence of a weak solution to the Euler equation, namely we perform first a priori 
estimates and then the estimates for the third derivatives.  

We take the scalar product of \eqref{sysob} and the velocity $\vv$ and integrate over $(0,t)\times \Omega$: 
$$\int_0^t \!\! \intd \pd{\vv}{t}\cdot \vv \dx \dt + \int_0^t \!\! \intd (\vv\cdot\nabla)\vv \cdot \vv  \dx\dt + \int_0^t \!\!\intd \nabla p\cdot \vv \dx\dt = \int_0^t \!\! \intd\div\S \cdot \vv \dx\dt,$$
which leads to
\be{appstep1v}{\12 \int_0^t \ddt \|\vv\|^2_{L^2} \dt= - \int_0^t \!\!\intd \S: \nabla \vv\dx\dt = -\int_0^t \!\! \intd \S:\D\dx\dt.}
Further, we take the scalar product of \eqref{sysod} by $\S$ and integrate over $(0,t)\times \Omega$:
$$
\int_0^t \!\!\intd\pd{\S}{t}:\S \dx\dt + \int_0^t \!\!\intd(\vv \cdot \nabla)\S:\S \dx\dt + \int_0^t \!\!\intd(\S\W -\W\!\S):\S \dx\dt  = \int_0^t \!\! \intd \D:\S \dx\dt,$$
which implies that 
\be{appstep1S}{\12 \intt \ddt \|\S \|^2_{L^2}\dt = \intt\!\!\intd \D:\S\dx\dt.}
We sum up \eqref{appstep1v} and \eqref{appstep1S} and obtain
\be{appae1}{\|\vv(t) \|^2_{L^2} + \|\S(t) \|^2_{L^2} =  \|\vv(0) \|^2_{L^2} + \|\S(0) \|^2_{L^2} \text{ for all } t \in (0,T).}

Equation \eqref{ae1} gives uniform a priori estimates on $\vv$ and $\S$ in the following function spaces 
$$ \vv \in L^\infty(0,T;L^2),\quad \S \in  L^\infty(0,T;L^2).$$
Note that unlike the Oldroyd-type models the first a-priori estimate does not give any bound on the gradient of the velocity.

Next, we proceed to the estimates on the third spatial derivatives of $\vv$ and $\S$. We apply $D^3$ to \eqref{sysob} and take the scalar 
product with $D^3\vv$. After integration over $(0,T)\times\Omega$, we obtain
\begin{align*}
\intt\!\! \intd D^3\pd{\vv}{t}\cdot D^3\vv \dx\dt &+ \intt\!\! \intd D^3((\vv\cdot\nabla)\vv) \cdot D^3\vv  \dx\dt \\ &+ \intt\!\!\intd D^3\nabla p\cdot D^3\vv \dx\dt = \intt\!\! \intd D^3\div\S \cdot D^3\vv \dx\dt,\\
\12 \intt \ddt \|D^3 \vv\|^2_{L^2} \dt  &+ \intt\!\! \intd D^3((\vv\cdot\nabla)\vv) \cdot D^3\vv  \dx\dt = \intt\!\! \intd D^3\div\S \cdot D^3\vv \dx\dt.
 \end{align*}

Analogously, we apply $D^3$ to \eqref{sysod} and take the scalar product of the result with $D^3\S$. This, after integration over  $(0,T)\times\Omega$, leads to
\begin{align*}
\intt\!\!\intd D^3\pd{\S}{t}:D^3\S \dx\dt &+ \intt\!\!\intd D^3((\vv \cdot \nabla)\S):D^3\S \dx\dt \\ &+ \intt\!\!\intd D^3(\S\W -\W\!\S):D^3\S \dx\dt  = \intt\!\! \intd D^3\D:D^3\S \dx\dt,\\
\12 \intt \ddt \|D^3\S \|^2_{L^2}\dt &+ \intt \!\!\intd D^3((\vv \cdot \nabla)\S):D^3\S \dx\dt \\ &+ \intt \!\!\intd D^3(\S\W -\W\!\S):D^3\S \dx\dt  = \intt \!\! \intd D^3\D:D^3\S \dx\dt.
\end{align*}

The convective terms are estimated in the standard way. An important feature is the cancellation of the highest order term, as in \eqref{inertia2}: 
$$\intt \!\! \intd D^3((\vv\cdot\nabla)\vv) \cdot D^3\vv  \dx\dt \le C \intt \|D^3 \vv\|^3_{L^2} \dt,$$
$$\intt \!\!\intd D^3((\vv \cdot \nabla)\S):D^3\S \dx\dt \le  C \intt \|D^3\vv\|_{L^2}\|D^3\S\|^2_{L^2}\dt. $$
The obstacle is the co-rotational term $\intt \!\!\intd D^3(\S\W -\W\!\S):D^3\S \dx\dt$, or, more precisely, 
one of its terms, $\intt \!\!\intd (\S D^3\W -D^3\W\!\S):D^3\S \dx\dt$, which contains fourth derivatives of the velocity. 
Even if we restrict ourselves to planar flows, the difficulty persists. Indeed, in two dimensions we have
$$ \S =  \left(\begin{array}{cc}
S_{11} & S_{12} \\ 
S_{12} & S_{22}
\end{array}\right) ,\quad  \W = \12\left(\gradv -\gradv^T \right) = \12 \left(\begin{array}{cc}
0 & \pd{v_1}{x_2} - \pd{v_2}{x_1} \\ 
\pd{v_2}{x_1} - \pd{v_1}{x_2} & 0
\end{array}\right) = \12 \left(\begin{array}{cc}
0 & -\omega \\ 
\omega & 0\end{array}\right) ,$$
where $\omega = \pd{v_1}{x_2} - \pd{v_2}{x_1} $. Hence, we get
$$ \S\W = \12\omega\left(\begin{array}{cc}
S_{12} & -S_{11} \\ 
S_{22} & -S_{12}
\end{array}\right),\quad \W\!\S = \12\omega\left(\begin{array}{cc}
-S_{12} & -S_{22} \\ 
S_{11} & S_{12}
\end{array}\right),$$
$$ \S\W - \W\!\S = \12\omega\left(\begin{array}{cc}
2S_{12} & S_{22}-S_{11} \\ 
S_{22} - S_{11} & -2S_{12}
\end{array}\right) =: \12\omega \tilde{\S}. $$
Consequently, the most difficult term takes the form
$$ (\S D^3\W -D^3\W\!\S):D^3\S = \12 (D^3\omega) \tilde{\S}:D^3\S.$$

Despite our efforts to incorporate the equation for $\omega$ and its third derivatives, or to use the stream function, 
we were not able to control the last term by $D^3\vv$ and $D^3 \S$.

\section*{Acknowledgements}
J. Burczak was supported by MNiSW "Mobilność Plus" grant 1289/MOB/IV/2015/0. J.~M\'alek and P.~Minakowski acknowledge 
the support of the project P107/12/0121 financed by the Grant Agency of the Czech Republic. J.~M\'alek was also 
supported by the ERC-CZ project LL1202 financed by MSMT, Czech Republic. P.~Minakowski acknowledges also the support 
of National Science Centre of~Poland, grant number 2012/07/N/ST1/03369. 
The authors thank Piotr Gwiazda for general discussions on this topic.

\end{document}